\documentclass[letterpaper, 10 pt, conference]{ieeeconf}  

\IEEEoverridecommandlockouts                              
\usepackage{cite, graphicx, subfigure,url,setspace,algorithm}
\usepackage{amssymb}
\usepackage{bm}
\usepackage{amsmath}
\usepackage{CJK}

\newtheorem{proposition}{Proposition}
\newtheorem{remark}{Remark}
\newtheorem{corollary}{Corollary}
\newtheorem{Theorem}{Theorem}

\begin{document}
\title{An autoregressive (AR) model based stochastic unknown input realization and filtering technique }

\author{
Dan Yu,
\thanks{D. Yu is a Graduate Student Researcher, Department of Aerospace Engineering, Texas
  A\&M University, College Station}%
  \and
Suman Chakravorty \thanks{S. Chakravorty is an Associate Professor of Aerospace Engineering, Texas A\&M University, College Station}}

\maketitle
\begin{abstract}
This paper studies the state estimation problem of linear discrete-time systems with stochastic unknown inputs. 
The unknown input is a wide-sense stationary process while no other prior informaton needs to be known. We propose an autoregressive (AR) model based unknown input realization technique which allows us to recover the input statistics from the output data by solving an appropriate least squares problem, then fit an AR model to the recovered input statistics and construct an innovations model of the unknown inputs using the eigensystem realization algorithm (ERA). An augmented state system is constructed and the standard Kalman filter is applied for state estimation. A reduced order model (ROM) filter is also introduced to reduce the computational cost of the Kalman filter.  Two numerical examples are given to illustrate the procedure.
\end{abstract}
\section{Introduction}
In this paper, we consider the state estimation problem for systems with unknown stochastic inputs. The main contribution of our work is that when no prior information of the unknown inputs is known, we recover the statistics of  the unknown inputs from the measurements, and 
then construct an innovations model of the unknown inputs from the recovered statistics such that the standard Kalman filter can be applied for state estimation. The innovations model is constructed by fitting an autoregressive (AR) model to the recovered input correlation data from which a state space model is constructed using the balanced realization technique. The method is tested on stochastically perturbed heat and laminar flow problems. 

The problem of state estimation of systems with unknown inputs has received considerable attention over the past few decades. 
The unknown input observer (UIO) has been well established for  deterministic systems \cite{wang1975,bhatta, kudva}. Various methods of building full-order or reduced-order observers have been developed, such as \cite{hou1992, hui1993,darouach1994}.  Recently, sliding mode observers have been proposed for systems with unknown inputs  \cite{ssurvey}. The design parameters and matrices need to be well chosen to satisfy certain conditions in order for the observers to perform well. For systems without the ``observer matching"  condition being satisfied, a high-gain approach is proposed \cite{zak2010}.  The high-gain observers are used as  approximate differentiators to obtain the estimates of the auxiliary outputs. In the presence of measurement noise, the high-gain observer amplifies the noise, and extra care needs to be taken when designing the gain matrix. 

For stochastic systems, the problem of state estimation is known as unknown input filtering (UIF), and many UIF approaches are based on the Kalman filter \cite{uif2, uif3, uif1}. When the dynamics of the unknown inputs is available, for example, if it can be assumed to be a wide-sense stationary process with known mean and covariance, one common approach called Augmented State Kalman Filter (ASKF) is used, where the states are augmented with the unknown inputs \cite{uif4}. To reduce the computational complexity of ASKF, optimal two-stage and three-stage Kalman filters have been developed to decouple the augmented filter into two parallel reduced-order filters by applying a U-V transformation \cite{uif9,uif7,uif6}. When no prior information about the unknown input is available, an unbiased minimum-variance (UMV) filtering technique has been developed \cite{uif5, uif8}. 
The problem is transformed into finding a gain matrix such that the trace of the estimation error matrix is minimized. Certain algebraic constraints must be satisfied for the unbiased estimator to exist. In both the approaches above, the process noise is assumed to be white noise with known covariance.

In practice, there are many applications where the unknown inputs can be modeled as a stochastic process. For example, the state estimation of perturbed laminar flows is considered in \cite{fluidap1}. It shows that the external disturbances (as well as the sensor noise and initial conditions) can be modeled as unknown  stochastic inputs  which perturb the linearized Navier-Stoke equations. Thus, the state estimation problem of such system is transformed into the unknown input filtering problem with stochastic unknown inputs. Also, our work can be applied to identify the statistics of colored process noise. There is some research that considers the Kalman filtering with unknown noise covariances \cite{noiseap2,noiseap1}. The process noise is assumed to be white noise with unknown covariance, while in our approach, the process noise can be colored in time as well. There are also applications of our technique in signal processing, such as the wideband power spectrum estimation \cite{signalap1}, where the problem is to recover the unknown power spectrum of a wide-sense stationary signal from the obtained sub-Nyquist rate samples.

In this paper, we address the state estimation problem of systems with stochastic unknown inputs. The unknown inputs are assumed to be wide sense stationary, while no other information about the unknown inputs is known. We propose a new unknown input filtering approach based on system realization techniques. Instead of constructing the gain matrix which needs to satisfy certain constraints, we apply the standard Kalman filtering using the following procedure: 1) recover the statistics of the unknown inputs from the measurements by solving an appropriate least squares problem, 2) find a spectral factorization of unknown input process by fitting an autoregressive (AR) model, 3) construct an innovations model of the unknown inputs  via the eigensystem realization algorithm (ERA) \cite{JUANG} to the recovered input correlation data, and 4) apply the Augmented State Kalman Filter for state estimation. Different from existing methods, we construct a stochastic unknown input model  from sensor data,  which can be colored in time. To reduce the computational cost of the ASKF, we apply the Balanced  Proper Orthogonal Decomposition (BPOD) technique \cite{rowley1} to construct a reduced order model (ROM) for filtering. 

The main advantage of the AR model based algorithm we propose is that the performance of the algorithm is better than the ASKF, OTSKF and UMV algorithms when the unknown inputs can be treated as WSS processes with rational PSDs. The AR model based algorithm we propose constructs one particular realization of the true unknown input model, and the performance of the AR model based algorithm is the same as OTSKF when the assumed unknown input model used in OTSKF is accurate, and is better than UMV algorithm in the sense that the error covariances are smaller.  With the increase of the sensor noise, we have seen that the performance of AR  model based algorithm gets much better than the UMV algorithm. 

The paper is organized as follows. In Section \ref{Section 2}, the problem is formulated, and general assumptions are made about the system and the unknown inputs. In Section \ref{section 3},  the AR based unknown input realization approach is proposed. The unknown input statistics are recovered from the measurements, then a linear model is constructed using an AR model and the ERA is used to generate a balanced minimal realization of the unknown inputs. 
After an innovations model of the unknown inputs is constructed, the ASKF is applied for state estimation in Section \ref{Section 5}. Also, a ROM constructed using the BPOD is introduced to reduce the computational cost of Kalman filter. Section \ref{Section 4} presents two numerical examples that utilize the proposed technique.

\section{Problem Formulation}\label{Section 2}

Consider a complex valued linear time-invariant discrete time system:
\begin{eqnarray} \label{original system}
x_k = A x_{k-1} + B u_{k-1}, \nonumber \\
y_k = C x_k + v_k,
\end{eqnarray}
where $x_k \in \mathbb{C}^n$,  $y_k \in  \mathbb{C}^q$, $v_k \in  \mathbb{C}^q$, $u_k \in \mathbb{C}^p$  are the state vector, the measurement vector, the measurement white noise with known covariance, and the unknown stochastic inputs respectively. The process $u_k$  is used to  model the presence of the external disturbances, process noise, and unmodelled terms. Here,  $A \in  \mathbb{C}^{n \times n}$, $B \in  \mathbb{C}^{n \times p}$, $C \in  \mathbb{C}^{q \times n}$ are known.                                                                                                                                                                                                                                                                                                                                                                                                                                                                                                                                                                                                                                                                                                                                                                                                                                                                                                                                                                                                                                                                                                                                                                                                                                                                                                                                                                                                                                                                                                                                                                                                                                                                               

Denote $h_i = CA^{i-1}B, i = 1, 2, \cdots$ as the  Markov parameters of system (\ref{original system}). We use $x^*$ to denote the complex conjugate transpose of $x$, and $x^T$ to denote the transpose of $x$. 
Denote  $\bar{h}_i$ as the matrix $h_i$ with complex conjugated entries, and $h_i^* = (\bar{h}_i)^T$. $\| A\| = (\sum_{i, j =1}^n |a_{i,j}|^2)^{1/2}$ denotes the Frobenius norm of matrix $A$, and $\| x \|_2 = ( |x_1|^2 + |x_2|^2 + \cdots + |x_n|^2)^{1/2}$ denotes the Euclidean norm of vector $x$. 

The following assumptions are made about system (\ref{original system}):
\begin{itemize}
\item A1. $A$ is a stable matrix, and $(A, C)$ is detectable. 
\item A2. rank$(B) = p$, rank$(C) = q$, $p \leq q$ and rank $(CAB)$ =  rank $(B) =  p$.
\item A3. $u_k$ and $v_k$ are uncorrelated.
\item A4. We further assume that the unknown input $u_k$ can be treated as a WSS process:
\begin{eqnarray} \label{actual colored noise}
\xi_k = A_e \xi_{k-1} + B_e \nu_{k-1}, 
u_k = C_e \xi_k + \mu_k,
\end{eqnarray}
where $\nu_k$, $\mu_k$ are uncorrelated white noise processes.
\end{itemize}

\begin{remark}
A2 is a weaker assumption than the so-called ``observer matching"  condition used in unknown input observer design. The observer matching condition requires rank $(CB)$ = rank $(B) = p$, which in practice, may be too restrictive.  A2 implies that if there are $p$ inputs, then there should be at least $p$ controllable and observable modes.  A4 implies that $u_k$ is a WSS process with a rational power spectrum.
\end{remark}

In this paper, we consider the state estimation problem when the system (\ref{actual colored noise}), i.e., $(A_e, B_e, C_e)$ are unknown. Given the output data $y_k$, we want to construct an innovations model  for the unknown stochastic input $u_k$, such that the output statistics of the innovations model and  system (\ref{actual colored noise}) are the same. Given such a realization of the unknown input, we apply the standard Kalman filter for state estimation, augmented with the unknown input states.

\section{AR based Unknown Input Realization Technique}\label{section 3}
In this section, we propose an AR based unknown input realization technique which can construct an innovations model of the unknown inputs such that the ASKF can be applied for state estimation. First, a least squares problem is formulated based on the relationship between the inputs and outputs to recover the statistics of the unknown inputs. Then an AR model is constructed using the recovered input statistics, and a balanced realization model is then constructed using the ERA.

\subsection{Extraction of Input Autocorrelations via a Least Squares Problem} \label{section 3A}
Consider system  (\ref{original system}) with zero initial conditions, the output $y_k$ can be written as:
\begin{eqnarray}
y_k = \sum_{i=1}^{\infty} h_i u_{k-i} + v_k.
\end{eqnarray}

For a linear time-invariant (LTI) system, under assumption A1 that $A$ is stable, the output $\{y_k\}$ is a wide-sense stationary process when \{$u_k$\} is wide-sense stationary. From the definition of the autocorrelation function of a WSS process, the output autocorrelation  can be written as:
\begin{eqnarray}
R_{yy}(m) = E[y_k y_{k+m}^*] \nonumber \\
= \sum_{i=1}^{\infty} \sum_{j=1}^{\infty} h_i u_{k-i} u_{k+m-j}^* h_j^* + R_{vv}(m) \nonumber \\
=  \sum_{i=1}^{\infty} \sum_{j=1}^{\infty} h_i R_{uu}(m+i-j) h_j^* + R_{vv}(m),
\end{eqnarray}
where $m = 0, \pm 1, \pm 2, \cdots$ is the time-lag  between $y_k$ and $y_{k+m}$. Here, assumption A3  is used.

Notice that $R_{yy}(-m) \neq R_{yy}(m)$ when $\{y_k\}$  is a sequence of complex valued vectors. We denote $\hat{R}_{yy}(m) = R_{yy}(m) - R_{vv}(m)$, where $R_{vv}(m) = \Omega$ for $m=0$, and $R_{vv}(m) = 0$, otherwise. Therefore, the relationship between input  and output autocorrelation function is given by:
\begin{eqnarray}\label{out-in}
\hat{R}_{yy}(m) =  \sum_{i=1}^{\infty} \sum_{j=1}^{\infty} h_i R_{uu}(m+i-j) h_j^*.
\end{eqnarray}

For multiple input multiple output (MIMO) systems, $h_i$, $\hat{R}_{yy}(m)$, $R_{uu}(m)$ are matrices. To solve for the unknown input autocorrelations $R_{uu}(m)$, first we need to use a theorem from linear matrix equations \cite{direct, kronecker}. 
\begin{Theorem} \label{matrix equation}
Consider the matrix equation 
\begin{eqnarray} \label{matrix form}
AXB = C,
\end{eqnarray}
where $A$, $B$, $C$, $X$ are all matrices. If  $A \in \mathbb{C}^{m \times n} = (a_1, a_2, \cdots, a_n)$, where $a_i$ are the columns of $A$, then define $\text{vec}(A) \in \mathbb{C}^{mn \times 1}$ as: $\text{vec}(A) = \begin{pmatrix} a_1 \\ a_2 \\ \vdots \\ a_n \end{pmatrix}$.

The matrix equation (\ref{matrix form}) can be transformed into one vector equation:
\begin{eqnarray}\label{vector equation}
(B^T \otimes A) \text{vec}(X)= \text{vec}(C),
\end{eqnarray}
where $B^T \otimes A $ is the Kronecker product of $B^T $ and $A$. If $A$ is an $m \times n$ matrix and $B$ is a $p \times q$ matrix, then the Kronecker product $A \otimes B$ is the $mp \times nq$ block matrix:
\begin{eqnarray}
A \otimes B = \begin{pmatrix} a_{11} B &  a_{12} B & \cdots & a_{1n} B \\ \vdots & \vdots & \cdots & \vdots \\ a_{m1} B & a_{m2} B & \cdots & a_{mn}B \end{pmatrix}.
\end{eqnarray}
\end{Theorem}

By applying Theorem \ref{matrix equation}, (\ref{out-in}) can be written as:
\begin{eqnarray}\label{re_yuinf}
\underbrace{\text{vec} (\hat{R}_{yy} (m))}_{\in R^{q^2 \times 1}} =\sum_{i=1}^{\infty} \sum_{j=1}^{\infty}\underbrace{  \bar{h}_j \otimes h_i}_{\in R^{q^2 \times p^2}} \underbrace{\text{vec} (R_{uu}(m+i-j))}_{\in R^{p^2\times 1}},
\end{eqnarray}
where $\bar{h}_i$ denotes the matrix $h_i$ with complex conjugated entries, and $h_i^* = (\bar{h}_i)^T$.

Now, we estimate the unknown input autocorrelations by the following procedure. 
\subsubsection{Choose design parameter $M$}
Under assumption A1, i.e., the system is stable, the Markov parameters of the system (\ref{original system})  have the following property: $\| h_i \| \rightarrow 0$ as $i \rightarrow \infty$. 

We choose a design parameter $M$, such that (\ref{re_yuinf}) can be written as:
\begin{eqnarray}\label{re_yu}
\text{vec} (\hat{R}_{yy} (m)) =
\sum_{i=1}^{M} \sum_{j=1}^{M} \bar{h}_j \otimes h_i \text{vec} (R_{uu}(m+i-j)).
\end{eqnarray}
where $M$  varies with different systems and can be chosen as large as desired. 
\subsubsection{Choose design parameters $N_o$, $N_i$}
Under assumption A1 and A4, $\| R_{uu} (m) \| \rightarrow 0$, and $\| \hat{R}_{yy} (m) \|  \rightarrow 0$ as $m \rightarrow \infty$. As a standard method when computing a power spectrum from an autocorrelation function, we choose design parameters $N_i$ and $N_o$, such that  the input autocorrelations are calculated when $|m| \leq N_i$, and the output autocorrelations are calculated when $|m| \leq N_o$.  
The numbers $N_o$  and $N_i$  depend on the dynamic system and unknown inputs, and can be chosen as large as required. We have the following proposition.
\begin{proposition}\label{P1}
The relation $N_i \leq N_o$ holds, which implies that all significant input autocorrelations can be recovered from the output autocorrelations. 
\end{proposition}
\begin{proof}
The support of $\hat{R}_{yy}$ is limited to $(-N_o, N_o)$, thus, we have: 
$\hat{R}_{yy}(N_o+1) = 0. $
From (\ref{re_yuinf}),
\begin{eqnarray}
\text{vec}(\hat{R}_{yy}(N_o+1))= \sum_{i=1}^{\infty} \bar{h}_i \otimes h_i \text{vec}( R_{uu}(N_o+1)) \nonumber \\
+ \sum_{i=2}^{{\infty}}\bar{h}_{i-1} \otimes h_i \text{vec}(R_{uu}(N_o))+ \cdots.
\end{eqnarray}
If $N_i > N_o$, which means $R_{uu}(N_o + 1) \neq 0$, then it follows that $R_{yy}(N_o + 1)$ is also not negligible, which contradicts the assumption, and hence, as a consequence, $N_i \leq N_o$.
\end{proof}

Thus, the following equation is used for computation of the unknown input autocorrelations.
\begin{eqnarray}\label{re_yu_cut}
\text{vec} (\hat{R}_{yy} (m)) =
\sum_{i=1}^{M} \sum_{j=1}^{M} \bar{h}_j \otimes h_i \text{vec} \underbrace{(R_{uu}(m+i-j))}_{| m + i - j | \leq N_i}, \nonumber \\
|m| \leq N_o
\end{eqnarray}

\subsubsection{Solve the least squares problem}
We collect $2 N_o + 1$ output autocorrelations, and from the above assumptions, there are $2N_i + 1$ unknown input autocorrelations:
\begin{eqnarray}\label{autocoef}
\underbrace{ \begin{pmatrix} \text{vec}(\hat{R}_{yy}(-N_o)) \\\text{vec}(\hat{R}_{yy}(-N_o+1 )) \\ \vdots \\ \text{vec}(\hat{R}_{yy}(0)) \\ \text{vec}\hat{R}_{yy}(1)) \\ \vdots \\ \text{vec}(\hat{R}_{yy}(N_o)) \end{pmatrix}}_{\text{vec}(\hat{R}_{yy})} = C_{yu} \underbrace{ \begin{pmatrix} \text{vec}(R_{uu}(-N_i)) \\\text{vec}(R_{uu}(-N_i+1)) \\ \vdots  \\ \text{vec}(R_{uu}(0)) \\ \text{vec}(R_{uu}(1)) \\ \vdots   \\ \text{vec}(R_{uu}(N_i)) \end{pmatrix}}_{\text{vec}(R_{uu})},
\end{eqnarray}
where $C_{yu}$ is the coefficient matrix and can be calculated from (\ref{re_yu_cut}). 

Under assumption A1, A2 and A4, we have the following proposition.
\begin{proposition}\label{P2}
Equation (\ref{autocoef}) has a unique least squares solution $\hat{R}_{uu}(m),  m = \pm 1, \pm 2, \cdots, \pm N_i$ .
\end{proposition}
\begin{proof}
We partition the matrix $C_{yu}$ into three parts as $C_{yu} = \begin{pmatrix} C_t \\ C_m \\ C_b \end{pmatrix},$
where $C_{m}$ contains the $q^2(N_o - N_i) + 1, \cdots, q^2(N_o + N_i + 1)$ rows of $C_{yu}$ and can be expressed as: {\tiny\begin{eqnarray}
C_{m} = \begin{pmatrix} \displaystyle \sum_{j = 1}^M \bar{h}_j \otimes h_j & \displaystyle \sum_{j = 1}^{M-1} \bar{h}_j \otimes h_{j+1} & \cdots & \cdots \\ \displaystyle \sum_{j=  1}^{M-1} \bar{h}_{j+1} \otimes h_j & \displaystyle   \displaystyle \sum_{j=1}^M \bar{h}_j \otimes h_j & \cdots &  \cdots \\ \cdots & \cdots & \ddots & \cdots \\ \cdots & \cdots & \cdots &  \displaystyle \sum_{j=1}^M \bar{h}_j \otimes h_j \end{pmatrix}.
\end{eqnarray}}
In the following, we prove that $C_{m} \in \mathbb{C}^{q^2(2N_i + 1) \times p^2(2N_i+1)}$ has full column rank $p^2 (2N_i + 1)$ by induction. 

Let $N_i = 0$, then
\begin{eqnarray}
C_{m}(0) = \sum_{j= 1}^M \bar{h}_j \otimes h_j = (C V_{co} \otimes CV_{co}) 
(I + \Lambda_{co} \otimes \Lambda_{co} \nonumber \\
+ \cdots + \Lambda_{co}^{M-1} \otimes \Lambda_{co}^{M-1}) (U_{co}' B \otimes U_{co}'B),
\end{eqnarray}
where $\Lambda_{co}$ are the controllable and observable eigenvalues of $A$, and $(V_{co}, U_{co})$ are the corresponding right and left eigenvectors. Under the assumption A2, if rank $(CAB) = p$, and since 
$CAB = CV_{co} \Lambda_{co} U_{co}' B$, which implies that rank $(C_{m}(0)) = p^2$.

If rank $C_{m}( N_i -1 )$ has rank $p^2 (2N_i - 1)$, then consider $C_{m} ( N_i)$: 
\begin{eqnarray}
C_{m} (N_i) = \begin{pmatrix} C_{m}(0) & C_{12}& C_{13} \\ C_{21} & C_{m}(N_i - 1) & C_{23}\\ C_{31} & C_{32} & C_{m}(0) \end{pmatrix},
\end{eqnarray}
where $C_{12}, C_{13}, C_{21}, C_{23}, C_{31}, C_{32}$ are some matrices, and it can be proved that $C_{m}(N_i)$ has $p^2 + p^2( 2N_i -1) + p^2 = p^2(2N_i + 1)$ independent columns, and hence, rank $(C_{m}(N_i)) = p^2(2N_i + 1)$. 

Thus, by induction,  $C_{m}$ has full column rank, and hence, $C_{yu}$ has full column rank. Since $q \geq p$, it is an overdetermined system, so there exists a unique solution to the least squares problem.
\end{proof} 

\begin{remark}
The size of $C_{yu}$ is $q^2 (2N_o+1) \times p^2 (2N_i + 1)$ and it would be large when $p$ and $q$ increase, and hence, large scale least squares  problem needs to be solved for systems with large number of inputs/outputs. For example, a modified conjugate gradients method \cite{conjugate}  could be used as follows.

The least squares problem need to be solved is:
\begin{eqnarray}
\text{vec} (\hat{R}_{yy} )= C_{yu} \text{vec} (R_{uu}),
\end{eqnarray}
and multiply $C_{yu}^*$ on both sides:
\begin{eqnarray}
C_{yu}^* \text{vec} (\hat{R}_{yy})= C_{yu}^* C_{yu} \text{vec} (R_{uu}). 
\end{eqnarray}
If we denote $ L_s = C_{yu}^* \text{vec} (\hat{R}_{yy})$, $\bar{x} = \text{vec} (R_{uu})$, and $C_s = C_{yu}^*C_{yu}$, then $C_s = C_s^*$, and the problem is equivalent to solve the least squares problem for $\bar{x}$:
\begin{eqnarray}
C_s \bar{x}= L_s,
\end{eqnarray}
and a conjugate gradient method to solve this problem is summarized in Algorithm \ref {con_grad}.

\begin{algorithm}[!tb] 
\begin{enumerate}
\item For a least squares problem $C_s \bar{x} = L_s$, where $C_s = C_s^*$, $\bar{x}$ is unknown.
\item Start with a randomly initial solution $\bar{x}_0$.
\item $r_0 = L_s - C_s \bar{x}_0$, $p_0 = r_0$.
\item for $k = 0$, repeat
\item $\alpha_k = \frac{r_k^* r_k}{p_k^* C_s p_k} $,\\
$\bar{x}_{k+1} = \bar{x}_k + \alpha_k p_k$, \\
$r_{k+1} = r_{k} - \alpha_k C_s p_k $,\\
if  $r_{k+1}$ is sufficient small then exit loop. \\
$\beta_k = \frac{r_{k+1}^* r_{k+1} }{r_k^* r_k} $,\\
$p_{k+1} = r_{k+1} + \beta_k p_k $,\\
$ k = k+1$, \\
end repeat.
\item The optimal estimation  is $x_{k+1}$.
\end{enumerate}
\caption{Conjugate gradient algorithm}\label{con_grad}
\end{algorithm}
\end{remark}

Denote $R_{uu}(m)$ as the ``true" input autocorrelations, and $\Delta(m) = R_{uu}(m) - \hat{R}_{uu}(m)$ as the error of the input autocorrelations we extract, $\Delta (m)$ results from two design parameters: the choice of $M$ and $N_i$. We analyze the errors seperately, in the following.

\begin{proposition}\label{P3}
Denote $R_{uu}^M(m)$ as the input autocorrelations we extract by using $M$ Markov parameters of the dynamic system. We assume that $\| h_i \| \leq \delta , i > M$, where $\delta$ is small enough. The error of input autocorrelations  is: $\| \Delta_M (m) \|  \leq k_M \delta$, where $k_M$ is some constant.
\end{proposition}

The Perturbation theory \cite{eigp} is used to prove the above result, and the proof is shown in Appendix \ref{AP1}. 

\begin{remark}
Error analysis in the Fourier domain. 

The power spetral density is defined as: 
\begin{eqnarray}
S_{uu}( \omega) = \sum_{k = -\infty}^{ \infty} R_{uu}(k) e^{- j k \omega}, \\
S_{yy}( \omega) = \sum_{k = -\infty}^{\infty} \hat{R}_{yy}(k) e^{- j k \omega},
\end{eqnarray}
Thus, by substituting (\ref{out-in}), the relationship between the output power spectral density and input power spectral density is:
\begin{eqnarray}
S_{yy}( \omega) = \sum_{k = -\infty}^{\infty} (\sum_{i=1}^{\infty} \sum_{t=1}^{\infty} h_i R_{uu}(k + i -t) h_t^*) e^{- j k \omega}\nonumber \\
= \sum_{k = -\infty}^{\infty} (\sum_{i=1}^{M} \sum_{t=1}^{M} h_i R_{uu}(k + i -t) h_t^*) e^{- j k \omega} + \Delta S_{M}(\omega) \nonumber \\
= S_{yy}^M (\omega) + \Delta S_{M}(\omega),
\end{eqnarray}
where 
\begin{eqnarray}
\Delta S_{M} (\omega) = \sum_{k= -\infty}^{\infty} (\hat{R}_{yy}(k) - \hat{R}_{yy}^M(k)) e^{-j k \omega} = \nonumber \\
 \sum_{k = -\infty}^{\infty} h_{M+1} R_{uu}(k) h_{M+1}^*  e^{-j k \omega} \nonumber \\
+ \sum_{k = -\infty}^{\infty} h_{M+1} R_{uu}(k) h_1^* e^{-j(k-M) \omega } + \cdots \nonumber \\
= h_{M+1} S_{uu}(\omega) h_{M+1}^* + h_{M+1} S_{uu}(\omega) e^{j M \omega} h_1^* + \cdots. 
\end{eqnarray}
Thus, $ \| \Delta S_{M} (\omega)\|  \leq k_1 \delta$, where $k_1$ is some constant. Hence, the truncation error by using $M$ Markov parameters  can be seen to be a small perturbation in the frequency domain. 
\end{remark}

\begin{proposition}\label{P4}
Denote $R_{uu}^N (m)$ as the input autocorrelations we extract under assumption $\| R_{uu}(m) \| \leq \delta, |m| > N_i$, and $\| \hat{R}_{yy} (m) \| \leq \delta, |m| > N_o$ where $\delta$ is small enough. The errors resulting from this assumption is $\| \Delta_N(m) \| \leq k_N \delta$, where $k_N$ is some constant.
\end{proposition}

The proof is shown in Appendix \ref{AP2}.

\begin{remark}
Error analysis in frequency domain:
\begin{eqnarray}
S_{yy}(\omega) = \sum_{k=-\infty}^{\infty} (\sum_{i=1}^{\infty} \sum_{t=1}^{\infty} h_i \underbrace{R_{uu}(k + i -t)}_{|k+i-t| \leq N_i} h_t^*) e^{-j k \omega} \nonumber \\
+  \sum_{k=-\infty}^{\infty} (\sum_{i=1}^{\infty} \sum_{t=1}^{\infty} h_i \underbrace{R_{uu}(k + i -t)}_{|k+i-t| > N_i} h_t^*) e^{-j k \omega} \nonumber \\
= S_{yy}^N(\omega)  + \Delta S_N(\omega),
\end{eqnarray}
where 
\begin{eqnarray}
\| \Delta S_N(\omega) \| \leq   \sum_{k=-\infty}^{\infty} (\sum_{i=1}^{\infty} \sum_{t=1}^{\infty} \|h_i\| \times  \delta \times \| h_t^*\|) e^{-j k \omega} \| \leq k_2 \delta, \nonumber
\end{eqnarray}
where $k_2$ is some constant.
\end{remark}

Under the assumptions A1-A4, the following proposition considers the total errors of input autocorrelations we recover.
\begin{proposition}\label{P5}
Denote $\hat{R}_{uu} (m)$ as the input autocorrelation function we estimate from the output autocorrelations, and let $\Delta (m) = R_{uu}(m) - \hat{R}_{uu} (m)$ be the error between the estimated input autocorrelation and the ``true" input autocorrelation.  We assume that $\| h_i \| \leq \delta , i > M$, $\| R_{uu}(m) \| \leq \delta, |m| > N_i$, and $\| \hat{R}_{yy} (m) \| \leq \delta, |m| > N_o$ where $\delta$ is small enough.
Then $\| \Delta (m) \| \leq k \delta$, where $k$ is some constant.
\end{proposition}

Proposition \ref{P3} and \ref{P4} are used for the proof, and the proof is shown in Appendix \ref{AP3}. The results above show that if $M$, $N_i$, $N_o$ are chosen large enough, the errors in estimating the input autocorrelations can be made arbitrarily small.

\subsection{Construction of the AR Based Innovations Model}
After we extract the input autocorrelations from the output autocorrelations, we want to construct a system which will generate the same statistics as the ones we recovered in Section \ref{section 3A}. If assumption A4 is satisfied, i.e., $\{u_k \}$ is WSS with a rational power spectrum,  the power spectrum of $u_k$ is continuous, and can be modelled as the output of a casual linear time invariant system driven by white noise \cite{stoc1985}. Such system can be constructed by using an autoregressive moving average (ARMA) model, and in practice, a MA model can often be approximated by a high-order AR model, and thus, with enough coefficients, any stationary process can be well approximated by using either AR or MA models (Chapter 9, \cite{ARMA}), and in this paper, we use an AR model to fit the data. In an AR model, the time series can be expressed as a linear function of its past values, i.e.,
\begin{eqnarray}\label{relation_ar}
u(k) = \sum_{i=1}^{M_i} a_i u(k-i) + \epsilon(k),
\end{eqnarray}
where $\epsilon(k)$ is white noise with distribution $N(0, \Omega_r)$, $M_i$ is the order of the AR model, and  $a_i, i=1, 2, \cdots, M_i$ are the coefficient matrices. For a vector autoregressive model with complex values, the Yule-Walker equation \cite{yule} which is used to solve for the coefficients needs to be modified. The modified Yule-Walker equation can be written as:
\begin{eqnarray}\label{yule-walker}
 \begin{pmatrix}R_{uu}(-1) &  R_{uu}(-2) &  \cdots & R_{uu}(-M_i) \end{pmatrix}=
\begin{pmatrix} a_1^* \\ a_2^* \\ \cdots \\ a_{M_i}^*\end{pmatrix}^* \times \nonumber \\
\begin{pmatrix} R_{uu}(0) & R_{uu}(-1) & \cdots & R_{uu}(1-M_i) \\
			R_{uu}(1) & R_{uu}(0) & \cdots & R_{uu} (2-M_i) \\
 			\vdots & \vdots & \vdots & \vdots \\
			R_{uu}(M_i-1) & R_{uu}(M_i-2) & \cdots & R_{uu}(0)  \end{pmatrix}.
\end{eqnarray}

Equation (\ref{yule-walker}) is used to solve for the coefficient matrices $a_i, i=1, 2, \cdots, M_i$. The covariance of the residual white noise $\epsilon(k)$ can be solved using the following equation:
\begin{eqnarray}\label{cov}
 R_{\epsilon \epsilon}(m)  = R_{uu}(m) - \sum_{i=1}^{M_i} \sum_{j=1}^{M_i} a_i R_{uu}(m+ i - j) a_j^*,
\end{eqnarray}
where $ \Omega_r = R_{\epsilon \epsilon}(0)$. The balanced minimal realization for the AR model (\ref{relation_ar}) can be expressed as:
\begin{eqnarray} \label{open loop}
\eta_k = A_n \eta_{k-1} + B_n u_{k-1}, \nonumber \\
{u}_k = C_n \eta_k + \epsilon_k,
\end{eqnarray}
where $(A_n, B_n, C_n)$ are solved by using the ERA technique \cite{JUANG} with $a_i, i=1, \cdots, M_i$ as the Markov parameters of the  system. A brief description of the ERA is given in Appendix \ref{AP4}.

Equation (\ref{open loop}) is equivalent to:
\begin{eqnarray}\label{closed loop}
\eta_k = (A_n+B_nC_n )\eta_{k-1} + B_n \epsilon_{k-1}, \nonumber \\
{u}_k = C_n \eta_k + \epsilon_k,
\end{eqnarray}
where $\epsilon_{k}$ is white noise with covariance $\Omega_r$. We make the following remark. 
\begin{remark}\label{era_stable}
We need to find a stable $A_n + B_nC_n$ in (\ref{closed loop}). In practice, we calculate the Markov parameters of system (\ref{closed loop}) using $a_i, i = 1, \cdots, M_i$ first, and then use the ERA for the state space realization.  If the Markov parameters of system (\ref{closed loop}) are $\hat{a}_i, i = 1, \cdots, M_i$, then $\hat{a}_1 = C_n B_n = a_1, \hat{a}_2 = C_n (A_n + B_nC_n) B_n = a_2 + a_1 a_1, \cdots $. As we explained before, for a WSS process with rational power spectrum, from \cite{stoc1985} , we can always find a stable realization $(A_n + B_nC_n, B_n, C_n)$. 
\end{remark}

By using the Cholesky Decomposition, we can find a unique lower triangular matrix $P$ such that:
\begin{eqnarray}
\Omega_r = PP^*.
\end{eqnarray}

If $w_k$ is white noise with distribution $N(0,1)$, then $Pw_k$ would be white noise with distribution $N(0, \Omega_r)$. Thus, the innovation model we construct that has the same statistics as the unknown input system (\ref{actual colored noise}) is:
\begin{eqnarray}\label{final}
\eta_k = (A_n+B_nC_n )\eta_{k-1} + B_n Pw_{k-1}, \nonumber \\
{u}_k = C_n \eta_k + Pw_k,
\end{eqnarray}
where $w_k$ is a randomly white noise with standard normal distribution. 

Under assumption A4, we have the following proposition.
\begin{proposition}\label{P6}
Denote $\hat{R}_{uu}(m)$ as the input autocorrelations recovered from the measurements,
then $\hat{R}_{uu}(m)$ can be reconstructed exactly by using the innovations model (\ref{final}), i.e., $\tilde{R}_{uu}(m)  = \hat{R}_{uu}(m)$, where $\tilde{R}_{uu}(m)$ is the input autocorrelations of the realization of system (\ref{final}).
\end{proposition}

From Proposition \ref{P5} and \ref{P6}, under the same assumptions, the following corollary immediately follows.
\begin{corollary}
Denote $u_k$ as the actual unknown input process, and $R_{uu}(m)$ as the actual input autocorrelation function. Then $\| \tilde{R}_{uu} (m) - R_{uu} (m) \| \leq k_a \delta$, where $k_a$ is some constant, when $\delta$ is small enough. System (\ref{final}) is an innovations model for the unknown input $u_k$. 
\end{corollary}

The procedure of constructing the innovations model is summarized in Algorithm \ref{unknown_algo}.
\begin{algorithm}[!tb] 
\caption{AR model based unknown input realization technique}\label{unknown_algo}
\begin{enumerate}
\item {Choose a finite number $N_o$, compute output autocorrelation function $R_{yy}(m)$ by using measurements $y_k$, $|m| \leq N_o$.}
\item {Choose a finite number $M$, construct the coefficient matrix $C_{yu}$ from (\ref{re_yu_cut}).}
\item {Choose a finite number $N_i$, solve the least squares problem (\ref{autocoef}) for unknown input autocorrelation function $R_{uu}(m)$, $|m| \leq N_i$.}
\item {Construct an AR model for the unknown input $u(k) = \sum_{i=1}^{M_i} a_i u(k-i) + \epsilon(k)$, find the coefficient matrices $a_i, i = 1, 2, \cdots M_i$ by solving the modified Yule-Walker equation (\ref{yule-walker}).}
\item {Find the covariance $\Omega_r$ of $\epsilon(k)$ by solving (\ref{cov}).}
\item {Construct the state space representation (\ref{open loop}) for the AR model using ERA.}
\item {Find a unique lower triangular matrix $P$ such that $\Omega_r = P P^*$, and construct an innovations model as in (\ref{final}).}
\end{enumerate}
\end{algorithm}

\begin{remark}
For real valued system, we can save the computation by using the properities of autocorrelation functions:
\begin{eqnarray}
R_{u_iu_i} (-m) = R_{u_i u_i}(m),  \nonumber \\
R_{u_iu_j} (-m) = R_{u_j u_i}(m),  i \neq j
\end{eqnarray}
Thus, we only need to collect $N_o+1$ output autocorrelations and have $p^2(N_o+1)$ equations with $q^2(N_i +1)$ unknowns in (\ref{autocoef}). 
\end{remark}
\begin{remark}
A generalization to the joint state and unknown input estimation. 

When the unknown inputs affect both the states and outputs, i.e. 
\begin{eqnarray}\label{joint}
x_{k+1} = A x_k + B u_k,  \nonumber \\
y_k = C x_k + D u_k + v_k, 
\end{eqnarray}
where $u_k$ is the stochastic unknown input, $v_k$ is the measurement noise.  The solution $y_k$ can be written as:
\begin{eqnarray}
y_k = \sum_{i=1}^M h_i u_{k-i} + D u_k + v_k,
\end{eqnarray}
and the relationship between output autocorrelations and input autocorrelations is:
\begin{eqnarray}
R_{yy}(m) = \sum_{i=1}^M \sum_{j=1}^M h_i R_{uu}(m+i-j) h_j^* + R_{vv}(m) +  \nonumber \\
\sum_{i=1}^N h_i R_{uu} (m+i) D^* + \sum_{i=1}^N D R_{uu} (m-j) h_j^* + D R_{uu} (m) D^*,
\end{eqnarray}
which can also be formulated as a least squares problem (\ref{autocoef}), and an unknown input system may be realized following the same procedure as in Algorithm \ref{unknown_algo}.
\end{remark}
\section{Augmented State Kalman Filter and Model Reduction} \label{Section 5}
After we construct an innovations model for the unknown inputs,  we apply the standard Kalman filter on the augmented system with states augmented by the unknown input states. A ROM based filter is also constructed using the BPOD for reducing the computational cost of the resulting filter.
\subsection {Augmented State Kalman Filter}
The full order system can be represented by augmenting the states of the original system as:
\begin{eqnarray}\label{askf}
 \begin{pmatrix} x_{k+1} \\  \eta_{k+1} \end{pmatrix} = \begin{pmatrix} A & B C_n \\ 0 & A_n + B_n C_n \end{pmatrix}\begin{pmatrix} x_k \\ \eta_k \end{pmatrix} + \begin{pmatrix} B P \\ B_n P \end{pmatrix} w_k, \nonumber \\
 y_k = \begin{pmatrix} C & 0 \end{pmatrix} \begin{pmatrix} x_k \\ \eta_k \end{pmatrix}  +  v_k,
\end{eqnarray}
where $w_k$ is white noise with standard normal distribution. $v_k$ is white noise with known covariance. Thus, we may now use the standard kalman filter for state estimation of the augmented system (\ref{askf}).

\begin{remark}
The augmented state system (\ref{askf}) is stable and detectable. The eigenvalues of the augmented system (\ref{askf}) are the eigenvalues of $A$ and the eigenvalues of $A_n + B_n C_n$. From assumption A1, $A$ is stable, from Remark \ref{era_stable}, $A_n + B_n C_n$ is stable, and hence, the augmented system (\ref{askf}) is stable. From assumption A1, system (\ref{original system}) is detectable, and from the asymptotic stability of  matrix $A_n + B_n C_n$, (\ref{closed loop}) is also detectable, therefore, all the unobservable modes  in (\ref{askf}) are asymptotically stable, which implies that (\ref{askf}) is detectable. Thus, we may now use the standard Kalman filter for state estimation of the augmented system (\ref{askf}).
\end{remark}
\subsection{Unknown Input Estimation Using Model Reduction}

For large scale systems, we can use model reduction technique such as Balanced Proper Orthogonal Decomposition (BPOD) to construct a reduced order model (ROM) first, and then extract the input autocorrelations from the reduced order model. We apply the Kalman filter to the ROM to reduce the computational cost. A brief description of BPOD is given  in Appendix \ref{AP4}.  For a large scale system with a large number of inputs and outputs, we can also use the randomized proper orthogonal decomposition (RPOD) technique \cite{acc2015} for model reduction. 

The ROM system is extracted from the full order system using the BPOD and is denoted by:
\begin{eqnarray}
x_k = A_r x_{k-1} + B_r u_{k-1}, \nonumber \\
y_k = C_r x_k + v_k.
\end{eqnarray}
Let $\hat{h}_i = C_r A_r^{i-1}B_r, i = 1, 2, \cdots, M$ be the Markov parameters of the ROM. Then the relationship between input autocorrelations and output autocorrelations can be written as: 
\begin{eqnarray}
\hat{R}_{yy}(m) =  \sum_{i=1}^M \sum_{j=1}^M \hat{h}_i R_{uu}(m+i-j) \hat{h}_j^*.
\end{eqnarray}

Following the same procedure as in Algorithm \ref{unknown_algo}, we can now recover the input autocorrelations, and construct an innovations model which can generate the same statistics as the unknown inputs. The advantage of using model reduction is that for a large scale system, computing $\hat{h}_i = C_r A_r^{i-1} B_r$ is much faster than computing $h_i = C A^{i-1} B$ because of the reduction in the size of $A$. Also, the order of the ROM is much smaller than the order of the full order system, and thus the computational cost of using the Kalman filter is much reduced. Hence, even with the augmented states, the standard Kalman filter remains computationally tractable. 

\begin{remark}
To reduce the computational cost of the augmented states in Kalman filter, we can also use the existing optimal two-stage or three-stage kalman filtering technique \cite{uif9, uif6}, which decouple the augmented filter into two parallel reduced order filters. These techniques are preferable when the order of the innovations model is high, while the BPOD based ROM filter is preferable when the order of the dynamic system is high.
\end{remark}

\section{Computational Results} \label{Section 4}
We test the method on a one-dimensional heat equation and the perturbed laminar flow equation. We construct the unknown input system by using both the full order system as well as the ROM constructed by BPOD. We check the results by comparing the autocorrelation functions of the inputs, outputs and the states. Also, we show the state estimation using the Kalman filter. We define the relative error as:
\begin{eqnarray}
R_{relative} =\frac{ \|R_{true} - R_{es} \|}{\|R_{true}\|},
\end{eqnarray}
$R_{true}$ : actual output/input/state  autocorrelation function of the system\\
$R_{es}$   : estimated output/input/state  autocorrelation function 

In the following, we will show simulation results for the stochastically perturbed 1D heat equation and the laminar flow problem.
\subsection{Heat Equation}
The equation for heat transfer by conduction along a slab is given by the partial differential equation:
\begin{eqnarray}
\frac{\partial T}{\partial t} = \alpha \frac{\partial^2 T}{\partial x^2}+f, \nonumber \\
T|_{x= 0} = 0, \frac{\partial T}{\partial x}|_{x = L} = 0,
\end{eqnarray}
where $\alpha$ is the thermal diffusivity, $L = 1m$, and $f$ is the unknown forcing. There are two point sources  located at $x = 0.5m$ and $x = 0.6m$. 

The system is discretized using finite difference approach, and there are 50 grids which are equally spaced. To satisfy the observer matching condition in the UMV algorithm, we take  two measurements at $x = 0.5m$, $x = 0.6m$. The measurement noise is white noise with covariance $0.1 I_{2 \times 2}$. In the simulation, the unknown inputs are  generated using (\ref{actual colored noise}) with 
\begin{eqnarray}\label{parameter}
A_e = \begin{pmatrix} 0.3 & 0.5 \\ 0.4 & 0.2 \end{pmatrix}, B_e = C_e = I_{2 \times 2},
\end{eqnarray}
and $\nu_k =0, \mu_k \sim N(0, 10I_{2 \times 2}).$ The design parameters $M = 4000$, $N_i = 200$, $N_o = 2000$ are chosen as follows.  $M$ is chosen so that the Markov parameters $\|h_i \| \approx 0, i > M$. $N_i$ and $N_o$ are chosen by trial and error. First, we randomly choose a suitable $N_i$ and $N_o$, where $N_i \leq N_o$. Then we follow the AR based unknown input realization procedure, and construct the augmented state system (\ref{askf}). Given the white noise processes $w_k$, $v_k$ perturbing the system, we check the output statistics of the augmented state system (\ref{askf}). If the errors are small enough, we stop, otherwise, we increase the values of $N_i$ and $N_o$, and repeat  the same procedure until the errors are negligible. Notice that increasing $M$, $N_i$, $N_o$ would increase the accuracy of the input statistics we can recover, but also increases the computational cost. 

First, in Figure \ref{heat_inputf}, we show the comparison of the input correlations we recover with the actual input correlations. Since there are two inputs, thus, the cross-correlation function between input 1 and input 2 are also included.
\begin{figure}[!tb]
\centering
\includegraphics[scale=0.5]{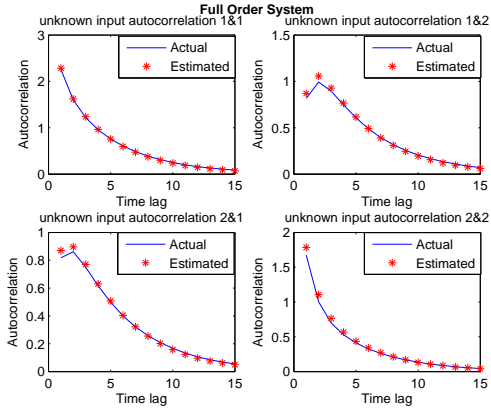}
\caption{Comparison of input autocorrelations}
\label{heat_inputf}
\end{figure}
It can be seen that the statistics of the unknown inputs can be recovered almost perfectly, and given the system perturbed by the unknown inputs innovations model we constructed, the statistics of the outputs and the states are almost the same as well. 

Next, we compare the performance of the unknown inputs constructed  using the ROM with the full order system. The full order system has 50 states, and the ROM has 20 states. The relative error of the input correlation is shown in Figure \ref{heat_input_relative}.  
\begin{figure}[!bt]
\centering
\includegraphics[width=3.3in]{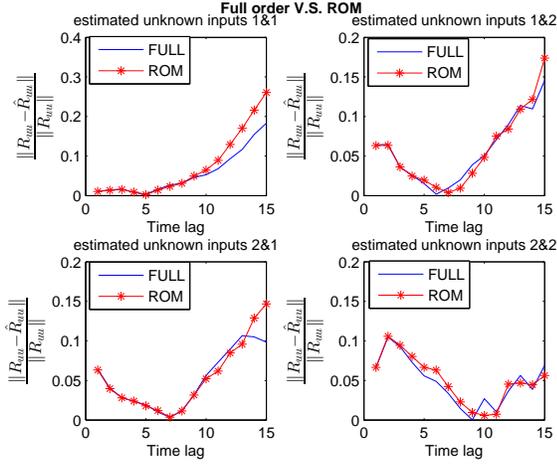}
\caption{Comparison of input autocorrelation relative error}
\label{heat_input_relative}
\end{figure}
We can see that the statistics reconstructed by using the ROM is not as accurate as using the full order system, however, the relative error is on the same scale, and hence, the computational cost is reduced without losing much accuracy. 

The state estimation using ROM is shown in Figure \ref{heat_state_filter}.  We randomly choose two states and show the comparison of the actual state with the estimated states. The state estimation error and $3 \sigma$ bounds are shown. It can be seen that the Kalman filter using the ROM performs well, and hence, for a large scale system, the computational complexity of ASKF can be reduced by using the BPOD. 
\begin{figure}[!tb]
\centering
\includegraphics[width=3.3in]{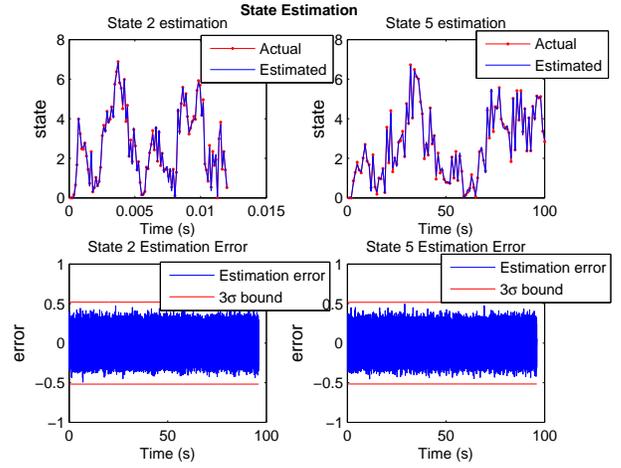}
\caption{Comparison of state estimation}
\label{heat_state_filter}
\end{figure}
\subsection{Comparison with OTSKF and UMV Algorithms}\label{section6}
Next, we compare the performances of the AR model based algorithm with OTSKF and UMV algorithms. The OTSKF and UMV algorithms we use can be found in \cite{rtskf1}. 

The assumed unknown input model used in the OTSKF is not the same as the true model, in particular, the system matrices of the input system are perturbed from the true values, the model used for OTSKF is: 
\begin{eqnarray}\label{wrong}
\eta_{k+1} = A_o \eta_k + v_k =  \begin{pmatrix} 0.4569  & 0.2768 \\ 0.2214& 0.4016 \end{pmatrix}\eta_k + v_k,
\end{eqnarray}
where $v_k \sim N(0, 10 I_{2 \times 2})$. Here, $A_o$ is chosen as follows. The eigenvalues of $A_e$  in (\ref{parameter}) are $0.7, -0.2$. We perturb the eigenvalues of $A_e$ with randomly generated numbers between $[-0.3, 0.3]$ and $[-0.8, 0.8]$ with uniform distribution respectively, and keep the eigenvectors same as the eigenvectors of $A_e$. The perturbed eigenvalues are $0.6783, 0.1802$.  We calculate the output statistics of (\ref{parameter}) and (\ref{wrong}), and we can see that the unknown input statistics used in OTSKF are perturbed by $5 \%$ about the true value. The estimation of the initial state $\bar{x}_0$ and covariance $\bar{P}_0$ in three algorithms are the same. 

Denote the average root mean square error(ARMSE) as:
\begin{eqnarray}
ARMSE = \frac{1}{n} \sum_{i = 1}^n \sqrt{\frac{\sum_{ k = 1}^n (\hat{x}_{i}(k)- x_{i}(k))^2}{n}},
\end{eqnarray}
where $\hat{x}_i(k)$ is the state estimate $\hat{x}_i$ at time $t_k$, and $x_i(k)$ is the true state $x_i$ at time $t_k$, where $i$ denotes the $i^{th}$ component of the state vector.

Suppose at the state component $x_i$, the measurement noise $v_k$ is a white noise with zero mean and covariance $\Omega_i$. We define a noise to signal  ratio (NSR): 
\begin{eqnarray}
NSR = \sqrt{\frac{{| \Omega_i | }}{(E[x_ix_i^*])}}.
\end{eqnarray}
We vary the measurement noise covariance $\Omega_i$, and for each $\Omega_i$, a Monte Carlo simulation of 10 runs is performed to compare the magnitude of the ARMSE  using AR model based algorithm with the OTSKF and UMV algorithms in Table \ref{tab1}.
\begin{table}[!tb]
\caption{Performances of the AR model based algorithm, OTSKF and UMV} \label{tab1}
\centering
\begin{tabular}{c|c|c|c|}
NSR  & AR model based & OTSKF & UMV \\ 
\hline
$0.2215 \%$ & 0.0036 & 0.0111 & 0.0033 \\
\hline 
$6.8704 \%$ & 0.0832 & 0.2418 & 0.0874\\
\hline 
$13.5171 \%$ & 0.1309& 0.3955& 0.1528 \\
\hline
$20.3456 \%$ & 0.3810  & 0.6516  & 0.4332 \\
\hline
$26.9467  \%$ & 0.4190   & 0.7141 & 0.5112 \\
\hline
\end{tabular}
\end{table}
The comparison is shown in Figure \ref{heat_comp}. It can be seen that the AR model based method performs the best. 
Note that when the assumed unknown input model used in OTSKF is not accurate, the performance of AR model based algorithm is much better while with  increase in the sensor noise, the performance of the AR model based  algorithm gets better than the UMV algorithm. It should also be noted that when  the sensors and the unknown inputs are non-collocated, the ``observer matching" condition is not satisfied, and hence, the UMV algorithm can not be used, while the OTSKF and the AR model based algorithm are not affected. 
\begin{figure}[!tb]
\centering
\includegraphics[height=4cm]{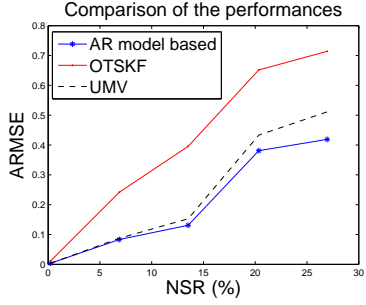}
\caption{Comparison of the performances}
\label{heat_comp}
\end{figure}

\subsection{Orr-Sommerfeld Equation}
Consider the three-dimensional flow between two infinite plates (at $y= \pm 1$) driven by a gradient in the streamwise $x$ direction. The mean velocity profile is given by $U(y) = 1 - y^2$. At each wavenumber pair $( \alpha, \beta)_{mn}$, the wall-normal velocity $v(x,y,z,t)$  and wall-normal vorticity $\eta(x, y, z, t)$ are:
\begin{eqnarray}
v(x, y, z, t) = \hat{v}_{mn}(y, t) e^{ i(\alpha x + \beta z)} ,\\
\eta(x, y, z, t) = \hat{\eta}_{mn}(y, t) e^{i(\alpha x + \beta z)}.
\end{eqnarray}

Denote 
\begin{eqnarray}
\hat{q}_{mn}(y, t) = \begin{pmatrix} \hat{v}_{mn}(y, t) \\ \hat{\eta}_{mn}(y,t) \end{pmatrix},
\end{eqnarray}
where $\hat{(.)} $ denotes the Fourier transformed variable, and $(.)_{mn}$ denotes the wavenumber pair $(\alpha, \beta)_{mn}$.

The evolution of the flow in Fourier domain can  be written as:
\begin{eqnarray}\label{orr-sommer}
\frac{d}{dt} M \hat{q}_{mn} + L \hat{q}_{mn} = T f(y,t),
\end{eqnarray}
where 
\begin{eqnarray}
M = \begin{pmatrix}
-\Delta & 0 \\
0 & I
\end{pmatrix},
\end{eqnarray}
\begin{eqnarray}
L = \begin{pmatrix}
-i \alpha U \Delta + i \alpha U^{''} + \Delta^2 / Re & 0\\
i \beta U^{'} & i \alpha U - \Delta / Re
\end{pmatrix}.
\end{eqnarray}

Operater $T$ transforms the forcing $f = (f_1, f_2, f_3)^T$ on the evolution equation for the velocity vector $(u, v, w)^T$ into an equivalent forcing on the $(v, \eta)^T$ system \cite{fluidap1},
\begin{eqnarray}
T = \begin{pmatrix}
i \alpha D & k^2 & i \beta D\\
i \beta & 0 & -i \alpha
\end{pmatrix},
\end{eqnarray}
where 
\begin{eqnarray}
k^2 = \alpha^2 + \beta^2, \\
\Delta = D^2 - k^2,
\end{eqnarray}
and $D$, $D^2$ represent the  first  and second order differentiation operators in the wall-normal direction. The forcing  $f(y,t)$ accounts for the nonlinear terms and the external disturbances via an unknown stochastic model.

The boundary conditions on $v$ and $\eta$ correspond to no-slip solid walls
\begin{eqnarray}
v(\pm 1)= Dv(\pm 1) = \eta(\pm 1) = 0.
\end{eqnarray}

System (\ref{orr-sommer}) can be discretized using Chebyshev polynomials, and in the simulation, we assume there are two unknown inputs and two measurements.\\

In the simulation, the design parameters $M=1000$, $N_i = N_o = 100$ are chosen by trial and error as explained before. 
The unknown input $f$ is assumed to be a colored noise generated by a third order linear complex system. The realization of the unknown inputs is a second order system. The measurement noise is white noise with covariance $0.1 I_{2 \times 2}$. 

First, we show the comparison of the input autocorrelations we recover with the actual input autocorrelations in complex plane. Since there are two inputs, thus, the cross-correlation function between input 1 and input 2 are also included in the input autocorrelations.
\begin{figure}[bt]
\centering
\includegraphics[scale=0.5]{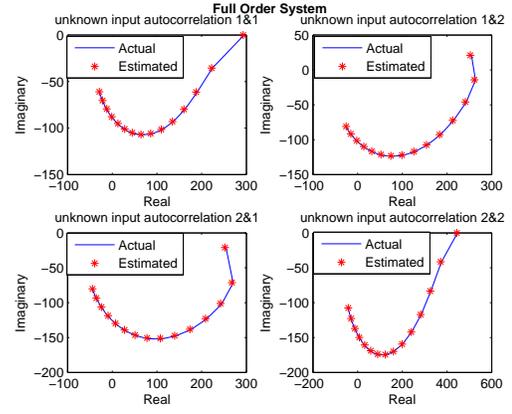}
\caption{Comparison of input autocorrelations}
\label{orr_inputf}
\end{figure}

Before we apply the ASKF for the state estimation, we compare the statistics of the states and outputs of the system perturbed by the unknown inputs we construct and the actual system. Fig. \ref{orr_outputf} shows the comparison between the estimated output autocorrelations and the actual autocorrelations. The comparison of the state autocorrelations is shown in Fig.\ref{orr_statef} for some randomly chosen states. 
\begin{figure}[tb]
\centering
\includegraphics[width=3.3in]{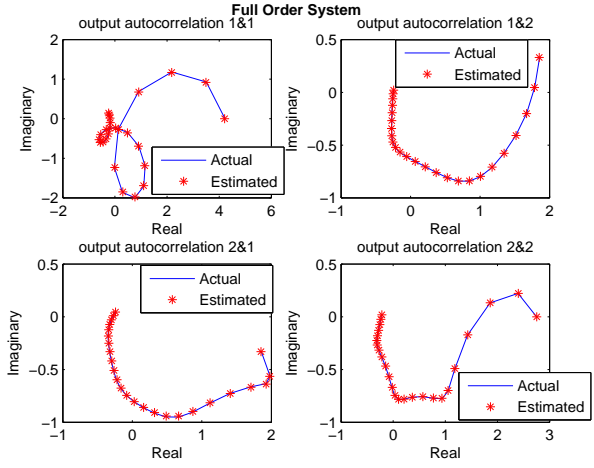}
\caption{Comparison of output autocorrelations}
\label{orr_outputf}
\end{figure}

\begin{figure}[bt]
\centering
\includegraphics[width=3.3in]{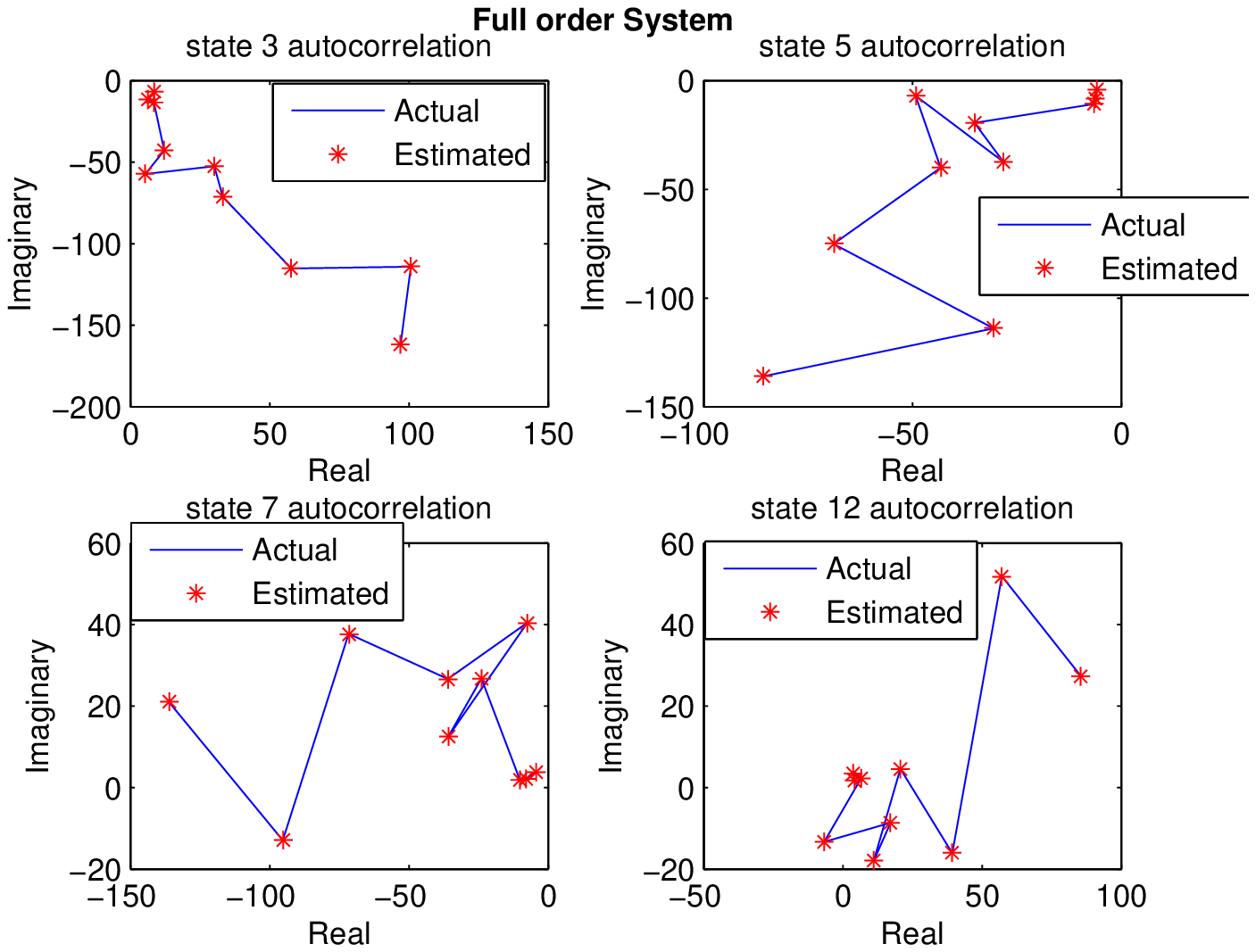}
\caption{Comparison of state autocorrelations}
\label{orr_statef}
\end{figure}

It can be seen that the statistics of the unknown inputs can be recovered almost perfectly, and given the system perturbed by the unknown inputs innovations model we constructed, the statistics of the outputs and the states are almost the same as well.

Next, we compare the performance of the unknown inputs constructed  using the ROM with the full order system. The full order system has 30 states, and the ROM has 15 states. The relative error of the input autocorrelation is shown in Fig. \ref{orr_input_relative}, and the comparison of the relative error of output autocorrelations is shown in Fig.\ref{orr_output_relative}.
\begin{figure}[bt]
\centering
\includegraphics[width=3.3in]{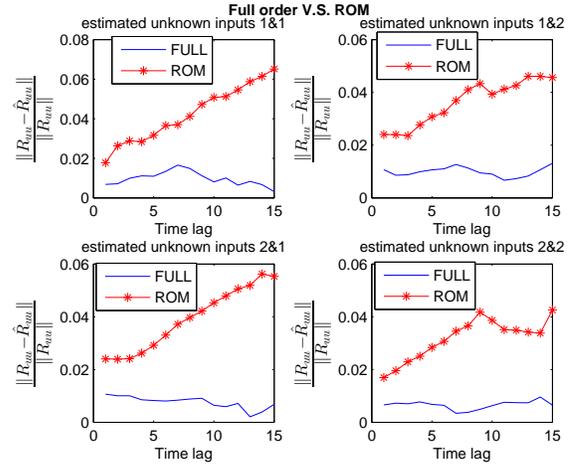}
\caption{Comparison of input autocorrelation relative error}
\label{orr_input_relative}
\end{figure}

\begin{figure}[bt]
\centering
\includegraphics[width=3.3in]{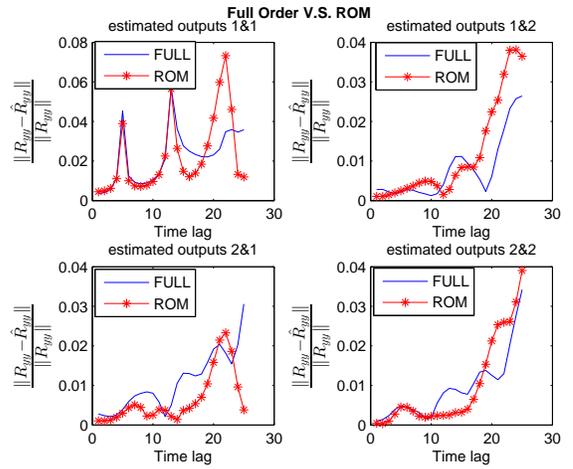}
\caption{Comparison of output autocorrelation relative error}
\label{orr_output_relative}
\end{figure}

The comparison of the relative error of state autocorrelations is shown in Fig. \ref{orr_state_relative}.
\begin{figure}[bt]
\centering
\includegraphics[width=3.3in]{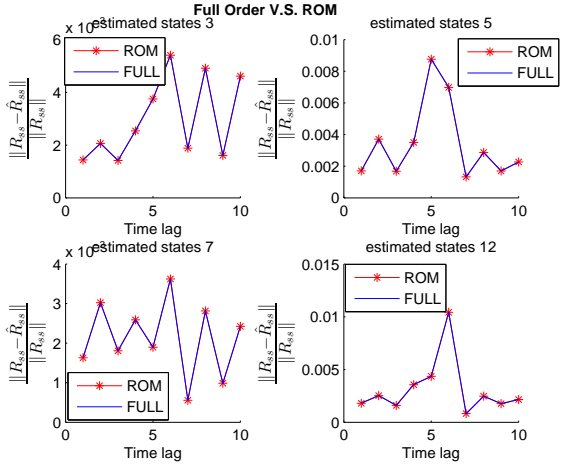}
\caption{Comparison of state autocorrelation relative error}
\label{orr_state_relative}
\end{figure}

We can see that the statistics reconstructed by using the ROM is not as accurate as using the full order system, however, the relative error is on the same scale, and hence, the computational cost is reduced without losing too much accuracy. 

The comparison of the state estimation using the ASKF is shown in Fig. \ref{orr_state_filter}. We randomly choose two states and show the comparison of the acutal state with the estimated states. The state estimation error and $3 \sigma$ bounds are shown. Since the error is complex valued, only the absolute value of the error is shown.
\begin{figure}[bt]
\centering
\includegraphics[width=3.3in]{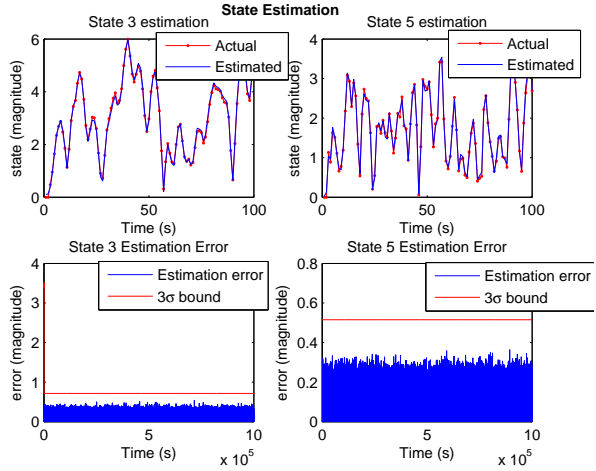}
\caption{State estimation using full order system}
\label{orr_state_filter}
\end{figure}

The state estimation using ROM is shown in Fig. \ref{orr_state_rom_filter}. It can be seen that the kalman filter using the ROM perform well, and hence, for a large scale system, the computational complexity of ASKF can be reduced by using the BPOD. 
\begin{figure}[bt]
\centering
\includegraphics[width=3.3in]{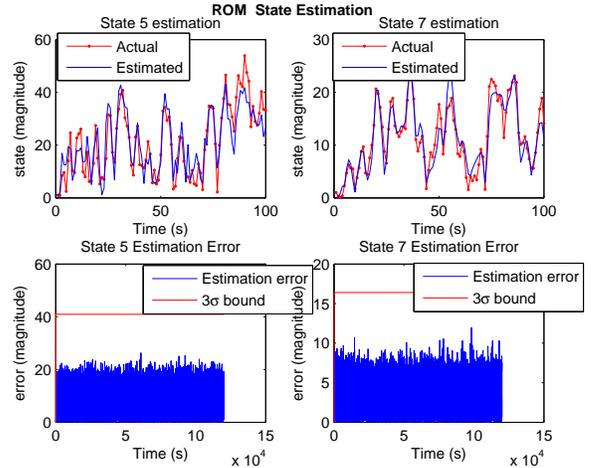}
\caption{State estimation using ROM}
\label{orr_state_rom_filter}
\end{figure}
\section{Conclusion}
In this paper, we have proposed a balanced unknown input realization method for the state estimation of system with unknown stochastic inputs. The unknown inputs are assumed to be a wide sense stationary process with a rational power spectrum, and no other prior information about the unknown inputs needs to be known. We recover the unknown inputs statistics from the output data using a least-squares procedure, and  then construct a balanced minimal realization of the unknown inputs using an AR model and the ERA technique. The recovered innovations model is used for state estimation, and the standard Kalman filter is applied on the augmented system. The next step in this process would require us to consider more complex realistic problems in fluid flow application, and cases where the unknown numbers of inputs/ outputs are large, and also cases where the locations of the inputs are unknown.
\appendices
\section{Proof of Proposition \ref{P3}} \label{AP1}
\begin{proof}
The output autocorrelation function using the first $M$ Markov parameters is:
\begin{eqnarray}
\hat{R}_{yy}^M(m) =  \sum_{i=1}^{M} \sum_{j=1}^{M} h_i R_{uu}(m+i-j) h_j^*.
\end{eqnarray}

Comparing with (\ref{out-in}), the output autocorrelation errors resulting from using $M$ Markov parameters is:
\begin{eqnarray}
\Delta_1(m) = 
\sum_{i=M+1}^{\infty} \sum_{j=1}^{M}  h_i  R_{uu}(m+i-j) h_j^*+\nonumber \\
\sum_{i=M+1}^{\infty} \sum_{j=M+1}^{\infty} h_i R_{uu}(m+i-j) h_j^*+\nonumber \\
 \sum_{i=1}^{M} \sum_{j=M+1}^{\infty}h_i R_{uu}(m+i-j) h_j^*.
\end{eqnarray}

From assumption A5, by choosing $M$ large enough, we have $ \| h_i\|  \leq \delta, i >  M$, where $\delta$ is small enough, thus,
\begin{eqnarray}
\| \Delta_1(m)\|  \leq \sum_{i=M+1}^{\infty} \sum_{j=1}^{M} \delta \times \|R_{uu}(m+i-j)\| \|h_j^*\| \nonumber \\
+ \sum_{i=M+1}^{\infty} \sum_{j=M+1}^{\infty} \delta \times \|R_{uu}(m+i-j)\|  \times \delta+\nonumber \\
+  \sum_{i=1}^{M} \sum_{j=M+1}^{\infty} \| h_i \|  \| R_{uu}(m+i-j)\| \times \delta \leq k_3  \delta,
\end{eqnarray}
 where $ k_3$ is some constant. 

Denote $C_{yu}$ as the ``true" coefficient matrix and $C_{yu}^M$ as the coefficient matrix using $M$ Markov papameters, we need to solve the least squares problem:
\begin{eqnarray}
\text{vec} (\hat{R}_{yy}) = C_{yu}^M \text{vec}(R_{uu}^M).
\end{eqnarray}
where $R_{uu}^M$ is the input autocorrelation we recover from using $M$ Markov parameters, and $\text{vec} (\hat{R}_{yy})$ is defined in (\ref{autocoef}).

Since $ \| \text{vec} (\hat{R}_{yy}(m)) - \text{vec} (\hat{R}_{yy}^M(m)) \|_2 = \|  \hat{R}_{yy}(m) -  \hat{R}_{yy}^M(m) \| = \| \Delta_1(m) \| \leq k_3 \delta $ , we have $\text{vec} (\hat{R}_{yy}(m)) = \text{vec} (\hat{R}_{yy}^M(m)) + \Delta_2(m)$, where  $\| \Delta_2 (m) \|_2  \leq k_3 \delta$, or equivalently
\begin{eqnarray}\label{err_ryy}
\text{vec} (\hat{R}_{yy}) = \text{vec} (\hat{R}_{yy}^M) + \Delta_2,
\end{eqnarray}

Consider (\ref{autocoef}), $\text{vec}(\hat{R}_{yy})$ and $\text{vec} (\hat{R}_{yy}^M) $ can be written as:
\begin{eqnarray}
\text{vec}(\hat{R}_{yy}) = C_{yu}   \text{vec}(R_{uu}), \nonumber \\
\text{vec}(\hat{R}_{yy}^M(m)) = C_{yu}^M  \text{vec} (R_{uu}),
\end{eqnarray}
Substitute into (\ref{err_ryy}), we have:
\begin{eqnarray} \label{coeff_err}
C_{yu} \text{vec}(R_{uu}) - C_{yu}^M \text{vec}(R_{uu}) = \Delta_2. 
\end{eqnarray}

Since $(C_{yu}^M)^{-1}$ exists, we have:
\begin{eqnarray}
\text{vec }(R_{uu}) - \text{vec}(R_{uu}^M )= (C_{yu}^M)^{-1} \Delta_2,
\end{eqnarray}
which means:
\begin{eqnarray}
\| \text{vec}(R_{uu}) - \text{vec}( R_{uu}^M) \|_2 \leq k_M \delta, 
\end{eqnarray}
where $k_M$ is some constant. Thus, we have $ \| \Delta_M(m)\| \leq k_M \delta$, where $k_M$ is some constant.
\end{proof}

\section{Proof of Proposition \ref{P4}}\label{AP2}
\begin{proof}
(\ref{re_yuinf}) can be seperated into two parts:
\begin{eqnarray}
\text{vec} (\hat{R}_{yy} (m))=\sum_{i=1}^{\infty} \sum_{j=1}^{\infty} \bar{h}_j\otimes h_i \text{vec} (\underbrace{R_{uu}(m+i-j)}_{ |m+i-j| \leq N_i}) \nonumber \\
+\sum_{i=1}^{\infty} \sum_{j=1}^{\infty} \bar{h}_j \otimes h_i \text{vec} (\underbrace{R_{uu}(m+i-j)}_{ |m+i-j| >  N_i}).
\end{eqnarray}

Thus, it can be written as:
\begin{eqnarray}
\text{vec}(\hat{R}_{yy} (m)) = \text{vec}(\hat{R}^N_{yy} (m)) + \Delta_4(m),
\end{eqnarray}
where 
\begin{eqnarray}
\| \Delta_4(m) \|_2 = \| \sum_{i=1}^{\infty} \sum_{j=1}^{\infty} \bar{h}_j \otimes h_i \text{vec} (\underbrace{R_{uu}(m+i-j)}_{ |m+i-j| >  N_i})\|_2 \nonumber \\
\leq \sum_{i=1}^{\infty} \sum_{j=1}^{\infty} \| \bar{h}_j \otimes h_i \|_2 \times \delta \leq k_4 \delta, 
\end{eqnarray}
where $k_4$ is some constant. $\| A \|_2$ denotes the induced 2-norm of matrix $A$.  Following the same procedure as in Proposition \ref{P3}, it can be proved that $\| \Delta_N(m) \| \leq k_N \delta$, where $k_N$ is some constant.
\end{proof}
\section{Proof of Proposition \ref{P5}} \label{AP3}
\begin{proof}
Denote output autocorrelation in (\ref{re_yu_cut}) as $\hat{R}_{yy}^c (m) $,
comparing (\ref{re_yu_cut}) with (\ref{re_yuinf}), the output autocorrelation error resulting from assumption A5 and A6 is:
\begin{eqnarray}
\text{vec}(\hat{R}_{yy}) - \text{vec}(\hat{R}_{yy}^c) = \Delta_2 +  \nonumber \\
\sum_{i=1}^{M} \sum_{j=1}^{M} \bar{h}_j \otimes h_i \text{vec} (\underbrace{R_{uu}(m+i-j)}_{|m + i-j| > N_i}) 
\leq \Delta_2 + \Delta_4.
\end{eqnarray}

Thus 
\begin{eqnarray}
\|\text{vec}(\hat{R}_{yy}) - \text{vec}(\hat{R}_{yy}^c) \|_2  \leq \| \Delta_2 \|_2 + \| \Delta_4 \|_2 \leq k_5 \delta,
\end{eqnarray}
where $k_5$ is some constant. Following the same precedure as in Proposition \ref{P3}, we can prove:
\begin{eqnarray}
\| \Delta(m) = R_{uu}(m) - \hat{R}_{uu}(m) \| \leq k \delta.
\end{eqnarray}
\end{proof}

\section{Brief Description of ERA and BPOD}\label{AP4}
The Eigensystem Realization Algorithm is summarized as follows.

Run inpulse response simulations of the linear system (\ref{original system}), and collect the snapshots of the outputs $y_k$ in the following patten:
\begin{eqnarray}
Y_1 = CB, Y_2 = CAB, \cdots, Y_k = C A^{k-1} B, 
\end{eqnarray}
where $CA^{k}B$ are known as Markov parameters. Construct a Hankel matrix $H(k)$
\begin{eqnarray}
H(k-1) = \begin{pmatrix} Y_k & Y_{k+1} & \cdots & Y_{k+ \beta -1} \\
				Y_{k+1} & Y_{k+2} & \cdots & Y_{k+ \beta} \\ 
				\vdots & \vdots & \cdots & \vdots \\
				Y_{k + \alpha -1} & Y_{k + \alpha} & \cdots & Y_{k + \alpha + \beta -2} \end{pmatrix}.
\end{eqnarray}

Solve the singular value decomposition (SVD) problem of $H(0)$, i.e.,
\begin{eqnarray}
H(0) = R \Sigma S^*.
\end{eqnarray}

Denote $\Sigma_n$ as the first $n$ non-zero singular value of $\Sigma$, and $R_n$, $S_n$ as the matrices formed by the first $n$ columns of $R$ and $S$ respectively. Then the realization for the ERA is:
\begin{eqnarray}
\hat{A} = \Sigma_n^{-1/2} R_n^* H(1) S_n \Sigma_n^{-1/2}, \nonumber \\
\hat{B} = \text{ first $p$ columns of } \Sigma_n^{1/2} S_n^* \nonumber \\
\hat{C} = \text{ first $q$ rows of  } R_n \Sigma_n ^{1/2}
\end{eqnarray}

The Balanced POD procedure using the impulse response of the primal and adjoint system and is summarized below.

Consider the linear system (\ref{original system}), and denote $B = [b_1, b_2, \cdots, b_p]$, $C = [c_1, c_2, \cdots, c_q]^*$. We collect the impulse response of the primal system by using $b_j$, $j = 1, 2, \cdots, p$, as initial conditions for the simulation of the system, 
\begin{eqnarray} \label{dps}
x_k = Ax_{k-1},
\end{eqnarray}

If we take $\alpha$ snapshots across the trajectories at time $t_1, t_2, \cdots, t_{\alpha}$, resulting an $N \times p\alpha$ matrix
\begin{eqnarray}
X= [x_1(t_1), \cdots , x_1(t_{\alpha}), \cdots, x_p(t_1), \cdots, x_p(t_{\alpha})],
\end{eqnarray}
where $x_j(t_k) $ is the state snapshot $x_k$ with $b_j$ as the initial condition.

Similarly, we use the transposed rows of the output matrix $c_i^*$, as the initial conditions for the simulations of the adjoint system $A^*$, 
\begin{eqnarray}
z_k = A^*z_{k-1},
\end{eqnarray}
and take $\beta$ snapshots across trajectories, leading to the adjoint snapshot ensemble $Y$,
\begin{eqnarray}
Y= [z_1(t_1), \cdots, z_1(t_{\beta}), \cdots, z_p(t_1), \cdots, z_p(t_{\beta})],
\end{eqnarray}
where $z_i(t_k)$ is the state snapshot $z_k$  with $c_i^*$ as the initial condition.

The Hankel matrix $H$ is constructed as:
\begin{eqnarray}
H= Y^*X.
\end{eqnarray}

Then we solve the SVD problem of the matrix $H$:
\begin{eqnarray}
H=Y^*X = U \Sigma V^*.
\end{eqnarray}

Assume that $\Sigma_1$ consists of  the first $r$ non-zero singular values of $\Sigma$, and $(U_1, V_1)$ are the corresponding left and right singular vectors from $(U,V)$, then the POD projection matrices can be defined as:
\begin{eqnarray}
T_r = XV_1 \Sigma_1^{-\frac{1}{2}}, \nonumber\\
T_l = YU_1 \Sigma_1^{-\frac{1}{2}},
\end{eqnarray}
and the reduced order model constructed using BPOD method is:
\begin{eqnarray}
\begin{cases}
A_r = T_l^* A T_r \\
B_r = T_l^* B \\
C_r = C T_r
\end{cases}
\end{eqnarray}

\bibliographystyle{IEEEtran}
\bibliography{IPOD_refs}

\begin{thebibliography}{10}
\providecommand{\url}[1]{#1}
\csname url@samestyle\endcsname
\providecommand{\newblock}{\relax}
\providecommand{\bibinfo}[2]{#2}
\providecommand{\BIBentrySTDinterwordspacing}{\spaceskip=0pt\relax}
\providecommand{\BIBentryALTinterwordstretchfactor}{4}
\providecommand{\BIBentryALTinterwordspacing}{\spaceskip=\fontdimen2\font plus
\BIBentryALTinterwordstretchfactor\fontdimen3\font minus
  \fontdimen4\font\relax}
\providecommand{\BIBforeignlanguage}[2]{{%
\expandafter\ifx\csname l@#1\endcsname\relax
\typeout{** WARNING: IEEEtran.bst: No hyphenation pattern has been}%
\typeout{** loaded for the language `#1'. Using the pattern for}%
\typeout{** the default language instead.}%
\else
\language=\csname l@#1\endcsname
\fi
#2}}
\providecommand{\BIBdecl}{\relax}
\BIBdecl

\bibitem{wang1975}
S.-H. Wang, E.J.Davison, and P.~Dorato, ``Observing the states of systems with
  unmeasurable disturbances,'' \emph{IEEE Transactions on Automatic Control},
  vol. 20,No.5, pp. 716--717, 1975.

\bibitem{bhatta}
S.~Bhattacharyya, ``Observer design for linear systems with unknown inputs,''
  \emph{IEEE Transactions on Automatic Control}, vol. AC-23,No.3, pp. 483--484,
  1978.

\bibitem{kudva}
P.~Kudva, N.Viswanadham, and A.~Ramakrishna, ``Observers for linear systems
  with unknown inputs,'' \emph{IEEE Transactions on Automatic Control}, vol.
  25,No.1, pp. 113--115, 1980.

\bibitem{hou1992}
M.~Hou and P.~Muller, ``Design of observers for linear systems with unknown
  inputs,'' \emph{IEEE Transactions on Automatic Control}, vol. 37, No.6, pp.
  871--875, 1992.

\bibitem{hui1993}
S.~Hui and S.~H. Zak, ``Low-order state estimators and compensators for
  dynamical systems with unknown inputs,'' \emph{Systems $\&$ Control Letters},
  vol. 21, No.6, pp. 493--502, 1993.

\bibitem{darouach1994}
M.~Darouach, M.~Zasadzinski, and S.~Xu, ``Full-order observers for linear
  systems with unknown inputs,'' \emph{IEEE Transactions on Automatic Control},
  vol. 39, No.3, pp. 606--609, 1994.

\bibitem{ssurvey}
S.~K. Spurgeon, ``Sliding mode observers: a survey,'' \emph{International
  Journal of Systems Science}, vol. 39, No.8, pp. 751--764, 2008.

\bibitem{zak2010}
K.~Kalsi, J.~Lian, S.~Hui, and S.~H. Zak, ``Sliding-mode observers for systems
  with unknown inputs: A high-gain approach,'' \emph{Automatica}, vol. 46,
  Issue 2, pp. 347--353, 2010.

\bibitem{uif2}
M.~Darouach, M.~Zasadzinski, A.~B. Onana, and S.~Nowakowski, ``Kalman filtering
  with unknown inputs via optimal state estimation of singular systems,''
  \emph{International Journal of Systems Science}, vol. 26(10), pp. 2015--2028,
  1995.

\bibitem{uif3}
M.~Hou and R.~J. Patton, ``Optimal filtering for systems with unknown inputs,''
  \emph{IEEE Transactions on Automatic Control}, vol. 43, No. 3, pp. 445--449,
  1998.

\bibitem{uif1}
D.~Koenig and S.~Mammar, ``Reduced order unknown input kalman filter:
  application for vehicle lateral control,'' in \emph{Proceedings of American
  Control Conference}, 2003, pp. 4353--4358.

\bibitem{uif4}
C.-S. Hsieh, ``A unified framework for state estimation of nonlinear stochastic
  systems with unknown inputs,'' in \emph{Proceedings of 9th IEEE Asian Control
  Conference}, 2013.

\bibitem{uif9}
C.-S. Hsieh and F.-C. Chen, ``Optimal solution of the two-stage kalman
  estimator,'' \emph{IEEE Transactions on Automatic Control}, vol.~44, pp.
  194--199, 1999.

\bibitem{uif7}
S.~Kanev and M.~Verhaegen, ``Two-stage kalman filtering via structured
  square-root,'' \emph{Communications in information and systems}, vol. 5,
  No.2, pp. 143--168, 2005.

\bibitem{uif6}
F.~B. Hmida, K.~Khemiri, J.~Ragot, and M.~Gossa, ``Three-stage kalman filter
  for state and fault estimation of linear stochastic systems with unknown
  inputs,'' \emph{Journal of the Franklin Institute}, vol. 349, pp. 2369--2388,
  2012.

\bibitem{uif5}
S.~Gillijns and B.~D. Moor, ``Unbiased minimum-variance input and state
  estimation for linear discrete-time systems,'' \emph{Automatica}, vol.~43,
  pp. 111--116, 2007.

\bibitem{uif8}
C.-S. Hsieh, ``Extension of unbiased minimum-variance input and state
  estimation for systems with unknown inputs,'' \emph{Automatica}, vol.~45, pp.
  2149--2153, 2009.

\bibitem{fluidap1}
J.~Hepffner, M.~Chevalier, T.~R. Bewley, and D.~S. Henningson, ``State
  estimation in wall-bounded flow systems. part 1. perturbed laminar flows,''
  \emph{Journal of Fluid Mechanics}, vol. 534, pp. 263--294, 2005.

\bibitem{noiseap2}
R.~K. Mehra, ``On the identification of variances and adaptive kalman
  filtering,'' \emph{IEEE Transactions on Automatic Control}, vol. 15, No.2,
  pp. 175--184, 1970.

\bibitem{noiseap1}
J.~Dunik and M.~Simandl, ``Estimation of state and measurement noise covariance
  matrices by multi-step prediction,'' in \emph{Proceedings of the 17th IFAC
  World Congress}, 2008, pp. 3689--3694.

\bibitem{signalap1}
D.~D. Ariananda and G.~Leus, ``Compressive wideband power spectrum
  estimation,'' \emph{IEEE Transactions on signal processing}, vol. 60,No.9,
  pp. 4775--4789, 2012.

\bibitem{JUANG}
J.-N. Juang, \emph{Applied System Identification}.\hskip 1em plus 0.5em minus
  0.4em\relax Englewood Cliffs, NJ: Prentice Hall, 1994.

\bibitem{rowley1}
C.~W. Rowley, ``Model reduction for fluids using balanced proper orthogonal
  decomposition,'' \emph{International Journal of Bifurcation and Chaos},
  vol.~15, pp. 997--1013, 2005.

\bibitem{direct}
W.~E. Roth, ``On direct product matrices,'' \emph{Bulletin of the American
  Mathematical Society}, vol.~40, pp. 461--468, 1934.

\bibitem{kronecker}
H.~Neudecker, ``Some theorems on matrix differentiation with special reference
  to kronecker matrix products,'' \emph{Journal of the American Statistical
  Association}, vol.~64, pp. 953--963, 1969.

\bibitem{conjugate}
V.~Faber and T.~Manteuffel, ``Necessary and sufficient conditions for the
  existence of a conjugate gradient method,'' \emph{SIAM Journal on Numerical
  Analysis}, vol.~21, pp. 352--362, 1984.

\bibitem{eigp}
T.~Kato, \emph{Perturbation Theory for Linear Operators}.\hskip 1em plus 0.5em
  minus 0.4em\relax New York: Springer-Verlag, 1995.

\bibitem{stoc1985}
E.~Wong and B.~Hajek, \emph{Stochastic Processes in Engineering Systems}.\hskip
  1em plus 0.5em minus 0.4em\relax New York: Springer-Verlag, 1985.

\bibitem{ARMA}
M.~West and J.~Harrison, \emph{Bayesian Forecasting and Dynamic Models}.\hskip
  1em plus 0.5em minus 0.4em\relax New York: Springer-Verlag, 1989.

\bibitem{yule}
B.~Friedlander and B.~Porat, ``The modified yule-walker method of arma spectral
  estimation,'' \emph{IEEE Transactions on Aerospace and Electronic Systems},
  vol. AES-20, No.2, pp. 158--173, 1984.

\bibitem{acc2015}
D.~Yu and S.~Chakravorty, ``A randomized proper orthogonal decomposition
  technique,'' in \emph{Proceedings of American Control Conference}, 2015, p.
  in press.

\bibitem{rtskf1}
C.-S. Hsieh, ``Robust two-stage kalman filters for systems with unknown
  inputs,'' \emph{IEEE Transactions on Automatic Control}, vol. 45, No.2, pp.
  2374--2378, 2000.

\end{thebibliography}

\end{document}